\documentclass[12pt,a4paper,reqno]{amsart}
\usepackage{amsmath, amsfonts, amsbsy, amsthm, amssymb}
\usepackage{graphicx}
\usepackage{xcolor}
\usepackage{enumerate,float,indentfirst}
\usepackage[T2A]{fontenc}
\usepackage[english]{babel}

\numberwithin{equation}{section}

\theoremstyle{plain}
\newtheorem{thm}{Theorem}[section]
\newtheorem{prop}[thm]{Proposition}
\newtheorem{lem}[thm]{Lemma}
\newtheorem{cor}[thm]{Corollary}
\newtheorem{rem}[thm]{Remark}
\theoremstyle{remark}

\newtheorem{df}[thm]{Definition}

\DeclareMathOperator{\E}{\mathbb{E}}

\title[{Sample Covariance of Sparse  Matrices}]
      {On the Largest and the Smallest  Singular Value of Sparse Rectangular Random Matrices}
      
\author[F.~G\"otze]{F.~G\"otze}
\address{Friedrich G{\"o}tze\\
 Faculty of Mathematics\\
 Bielefeld University \\
 Bielefeld, Germany
}
\email{goetze@math.uni-bielefeld.de}

\author[A.~Tikhomirov]{A.~Tikhomirov}

\address{Alexander N. Tikhomirov\\
 Institute of Physics and Mathematics\\
 Komi Science Center of Ural Branch of RAS \\
 Syktyvkar, Russia;  }

\email{tikhomirov@ipm.komisc.ru}

     
\date{\today}

\begin{document}

\begin{abstract}{
{We derive estimates for the largest and  smallest} singular values of sparse rectangular $N\times n$ random matrices, assuming $\lim_{N,n\to\infty}\frac nN=y\in(0,1)$.
We consider 
{a model with sparsity parameter $p_N$ such that $Np_N\sim \log^{\alpha }N$ for some $\alpha>1$, and assume that the moments of the matrix elements satisfy the condition 
$\E|X_{jk}|^{4+\delta}\le C<\infty$.  We assume also that  the entries of matrices we consider are truncated at} the level $(Np_N)^{\frac12-\varkappa}$ 
with $\varkappa:=\frac{\delta}{2(4+\delta)}$.
}\end{abstract}
\maketitle
\section{Introduction}In the last five to ten years, significant progress has been made in studying the 	asymptotic behavior of the spectrum of sparse random matrices. A typical example of such matrices is the incidence matrix of a random graph. Thus, for Bernoulli matrices Konstantin Tikhomirov obtained exact asymptotics for the probability of singularity, see~\cite{Kostya}; also, see~\cite{LitKos:20}. For the adjacency matrix of Erd\"os - Renyi random graphs, 
 H.-T. Yau and L. Erd\"os~\&~Co. proved a local semicircular law 
and investigated the behavior of the largest and the smallest singular values and  as well as eigenvector statistics, see the papers  of \cite{YauBurg,YauKnowless} and the literature therein. 
In particular  for adjacency matrices of regular graphs, local limit theorems 
and the behavior of extremal eigenvalues were investigated by  H.-T. Yau and co-authors \cite{DauKnowYau}.  For non-Hermitian sparse random matrices M. Rudelson and K. Tikhomirov proved the circular law under unimprovable conditions on the probability of sparsity and the moments of distributions of the matrix elements (see \cite{RudTikh}).
J.O. Lee and J.Y. Hwang  
studied the spectral properties of sparse sample covariance matrices, which includes adjacency
matrices of the bipartite Erd\"os--Renyi graph model).  
In  \cite{Lee} the authors prove a local law for the eigenvalues density up to the
upper spectral edge assuming that sparsity probability $p$ has order $N^{-1+\varepsilon}$ for some $\varepsilon>0$ (here $N$ denotes the growing order of the matrix)
and entries of matrix $X_{ij}$ are i.i.d. r.v.'s such that (in our notations)
\begin{equation}\label{Lee}
\E|X_{11}|^2=1\text{ and }\E|X_{11}|^q\le {(Cq)^{cq}} \text{ for every }q\ge1.
\end{equation}
{They also prove the Tracy-Widom limit law  for the largest eigenvalues of sparse sample covariance matrices. However, in the  proof of the local Marchenko-Pastur law and the Tracy-Widom limit, they assume a priori that the result of \cite[Lemma 3.11]{Ding} holds for sparse matrices (see  \cite[Proposition 2.13]{Lee}), which includes, in particular, the boundedness of the largest singular value that is the operator norm) of a sparse matrix. They don't investigate  the smallest singular value of sparse rectangular matrices though.}

We 
derive  bounds for the { \it smallest} and the {\it largest} singular values of sparse rectangular random matrices assuming that the probability  $p_N$ 
 decreases in such a way that $Np_N\ge \log^{\frac2{\varkappa}}N$ for some $\varkappa>0$, and that the moment conditions are weaker than those in \eqref{Lee} 
 (see condition \eqref{moment}). 
 Our main result is devoted to the smallest singular value of a sparse rectangular random matrix from an ensemble  of {\it dilute} Wigner type matrices. 

Suppose $n\ge1$ and  $N>n$.  
Consider independent identically distributed zero mean random variables $X_{jk}$, $1\le  j\le N$, $1\le k\le n$ with $\E X_{jk}^2=1$ {( where the distribution of $X_{jk}$ may depend on $N$), which are independent of  a set of independent Bernoulli random variables} $\xi_{jk}$, $1\le  j\le N$, $1\le k\le 
n$, with $\E\xi_{jk}=p_N$. 
In what follows we shall simplify notation by denoting  $p=p_N$. We now introduce the following model of  dilute sparse matrices as a sequence of random matrices of the following type
\begin{equation}\label{sparerectmat}
\mathbf X=(\xi_{jk}X_{jk})_{1\le j\le N, 1\le k\le n}.
\end{equation}
Denote by $s_1\ge \cdots\ge s_n$ the singular values of $\mathbf X$, and let $\mathbf Y=\mathbf X^*\mathbf X$ 
denote the  sample covariance matrix.

Put $y=y(N,n)=\frac nN$. We shall assume that $y(N,n)\to y_0<1$ as $N,n\to\infty$. {In what follows we shall 
 vary the parameter $N$ only.

\begin{thm}\label{lambda}Let  $\E X_{jk}=0$ and $\E |X_{jk}|^2=1$. 
Suppose that there exists a positive constant $C>0$ such that
\begin{equation}\label{delta}
\E|X_{jk}|^{4+\delta}\le C<\infty,
\end{equation}
for any $j,k\ge1$ and for some $\delta>0$. Suppose also that there exists a positive constant $B$, such that
\begin{equation}\label{np}
Np\ge B\log^{\frac3{2\varkappa}}N,
\end{equation}
where $\varkappa=\frac{\delta}{2(4+\delta)}$. 

Then for every $Q\ge1$ and $A>0$ there exists a constant $K=C(Q, \delta,\mu_{4+\delta}, A,B)$ such that
\begin{equation*}
\Pr\{\,s_1\ge K\sqrt{Np}\}\le CN^{-Q}{+N^2p\Pr\{|X_{11}|>A(Np)^{\frac12-\varkappa}\ln N}\}.
\end{equation*}

\end{thm} 
\begin{thm}\label{minmain}Let  $\E X_{jk}=0$ and $\E |X_{jk}|^2=1$. 
Suppose that
$$
\E|X_{11}|^{4}=\mu_4<\infty,
$$
 and  there exists a positive constant $B$, such that
\begin{equation}\label{np}
Np\ge B\log^{2}N.
\end{equation}

Then there exists  a constant $\tau_0>0$ such that for  every 
$\tau\le\tau_0$ , $Q\ge1$ and $K>0$ there exists a constant $C=C(Q, \mu_{4},K,B)$ with
\begin{equation*}
\Pr\{\,s_n\le \tau\sqrt{Np}\}\le CN^{-Q}{+\Pr\{s_1>K\sqrt{Np}\}}.
\end{equation*}

\end{thm}
These results  immediately imply  the following corollary. 
\begin{cor}Under conditions of Theorem \ref{lambda} there exist  
 a constant $\tau_0>0$  such that for any $\tau\le \tau_0$ and for any $A>0$ there exists a constant $C=C(A,\delta)$ depending on $A$ and $\delta$  such that the following inequality holds
\begin{equation*}
\Pr\{\,s_n\le \tau\sqrt{Np}\}\le CN^{-Q}{+N^2p\Pr\{|X_{11}|>A(Np)^{\frac12-\varkappa}\ln N\}}.
\end{equation*}
\end{cor}}
\begin{cor} \label{toy}Assume  the conditions of Theorem \ref{lambda}. In addition assume that there exists a constant $B$ such that for every $N\ge 1$
$$
p=p_N\ge B/\ln^4 N.
$$
Then
\begin{equation*}
\Pr\{\,s_1\ge K\sqrt{Np}\}\le CN^{-Q}+\frac C{\ln^{\delta} N}.
\end{equation*}

\end{cor}
\begin{proof} Applying Markov's inequality, we obtain
$$
\Pr\{|X_{11}|>A(Np)^{\frac12-\varkappa}\ln N\}\le \frac{\mu_{4+\delta}}{(Np)^2\ln^{4+\delta}N}.
$$
By the conditions  of Corollary \ref{toy}, we get
$$
\Pr\{|X_{11}|>A(Np)^{\frac12-\varkappa}\ln N\}\le \frac{\mu_{4+\delta}}{N^2B^{4+\delta}\ln^{\delta}N}.
$$
The result follows now immediately from theorem \ref{lambda}.
Thus, Corollary \ref{toy} is proved.
\end{proof}
We may consider  random variables $X_{ij}$ for $i=1,\ldots,N; j=1,\ldots,n$, with identical  distributions depending  on $N$.  In this case we have the following result.
\begin{cor}
In addition to conditions of Theorem \ref{lambda} assume that for any $q$ such that $4+\delta\le q\le C\log n$
\begin{equation}\label{moment}
\E|X_{11}|^q\le C_0^qq^q(Np)^{ q(\frac12-\varkappa)-2}.
\end{equation}
Then  for every $Q\ge1$ and $A>0$ there exist  constants $K=K(Q, \delta,\mu_{4+\delta}, A)$ and $C=C(Q, \delta,\mu_{4+\delta}, A)$ such that
\begin{equation*}
\Pr\{\,s_1\ge K\sqrt{Np}\}\le CN^{-Q}.
\end{equation*}
and there exists  a constant $\tau_0>0$ such that for  every 
$\tau\le\tau_0$ , $Q\ge1$  there exists a constant $C=C(Q, \delta,\mu_{4+\delta})$
\begin{equation}
\Pr\{\,s_n\le \tau\sqrt{Np}\}\le CN^{-Q}
\end{equation}

\end{cor}

\section{Proof of Theorem \ref{lambda}}
Let $\widetilde X_{ij}$ denote truncated random variables $X_{ij}$, i.e.
$$
\widetilde X_{ij}=X_{ij}\mathbb I\{|X_{ij}|\le A(Np)^{\frac12-\varkappa}\ln N\},
$$
where $\mathbb I\{B\}$ denotes the indicator of an event $B$. Let $\widetilde {\mathbf X}$ denote the matrix with entries $\xi_{ij}\widetilde X_{ij}$. By $\|\mathbf A\|$ we  denote the operator norm of a matrix $\mathbf A$.
First we estimate the spectral norm of the matrix $\E \widetilde{\mathbf X}$. Since $X_{ij}$ and $\xi_{ij}$ are identically distributed random variables we have
$$
\|\E\widetilde {\mathbf X}\|=np|\E\widetilde X_{11}|.
$$
By condition \eqref{delta}, we have
$$
|\E \widetilde X_{11}|=|\E X_{11}\mathbb I\{|X_{11}|>A(Np)^{\frac12-\varkappa}\ln N\}|\le \frac C{A^3(Np)^{\frac32+\varkappa}}.
$$
From here we get the bound
\begin{equation}\label{aa}
\|\E\widetilde{\mathbf X}\|\le CA^{-3}(Np)^{-\frac12-\varkappa}.
\end{equation}
We  consider now the centered and truncated random variables $\widehat X_{ij}=\widetilde X_{ij}-\E\widetilde X_{ij}$ for $i=1,\ldots N, j=1,\ldots n$, and the matrix $\mathbf {\widehat X}=(\xi_{ij}\widehat X_{ij}$). Let $\widehat s_1\ge\widehat s_2\ldots\ge \widehat s_n$ denote the singular values of the matrix $\widehat{\mathbf X}$ and resp. let $\widetilde s_1\ge\widetilde s_2\ldots\ge \widetilde s_n$ denote the singular values of the matrix $\widetilde{\mathbf X}$. 
Note that
\begin{align}
\Pr\{s_1\ne \widetilde s_1\}\le \Pr\{\mathbf X\ne\widetilde{\mathbf X}\}&\le \sum_{i=1}^N\sum_{j=1}^np\Pr\{\widetilde X_{ij}\ne X_{ij}\}\notag\\&=nNp\Pr\{|X_{11}|>A(Np)^{\frac12-\varkappa}\ln N\}
\end{align}

Furthermore, we have
\begin{equation}
 \widetilde s_1\le\widehat s_1+\|\E\widetilde {\mathbf X}\|.
\end{equation}
According to \eqref{aa} we may assume that
\begin{equation}
\|\E\widetilde {\mathbf X}\|\le \gamma\sqrt{Np}
\end{equation}
for sufficiently small $\gamma>0$.
We may write now
\begin{equation}
\Pr\{s_1>K\sqrt{Np}\}\le \Pr\{\widehat s_1>\frac12K\sqrt{Np}\}+N^2p\Pr\{|X_{11}|>A(Np)^{\frac12-\varkappa}\ln N\}
\end{equation}
Note that
\begin{align}
\widehat\sigma^2_n&=\E{\widehat X}_{11}^2=\E(\widetilde X_{11})^2-(\E\widetilde X_{11})^2\notag\\&=1-\E X_{11}^2\mathbb I\{|X_{11}|>A(Np)^{\frac12-\varkappa}\ln N\}-(\E X_{11}\mathbb I\{|X_{11}|>A(Np)^{\frac12-\varkappa}\ln N\})^2.
\end{align}
It is easy that
\begin{align}
|1-\sigma_n|\le
|1-\sigma_n^2|\le \frac{2\mu_{4+\delta}}{A^{2+\delta}(Np)^{(2+\delta)(\frac12-\varkappa)}}.
\end{align}
Without loss of generality we may assume that $\sigma_n\ge \frac12$. Consider now the matrix $\breve{\mathbf X}=\frac1{\sigma_n}\widehat{\mathbf X}$. Let $\breve s_1$ denote the largest singular value of {the}  matrix $\breve{\mathbf X}$. Then
\begin{equation}
\Pr\{\widehat s_1>K\sqrt{Np}\}\le \Pr\{{\breve s_1}>2K\sqrt{Np}\}.
\end{equation}

During the rest of the proof of Theorem \ref{lambda} we shall consider the matrix $\mathbf X$ with entries $\xi_{ij}X_{ij}$, $i=1,\ldots,N$ $j=1,\ldots,n$ satisfying the following conditions $(CI)$:
\begin{itemize}
	\item $\xi_{ij}$ are independent Bernoulli r.v.'s with $\E\xi_{ij}=p\,(=p_N)$; 
	\item $X_{ij}$ are i.i.d. r.v.'s for $1\le i\le N, 1\le j\le n$, such that
	$\E X_{11}=0$, $\E|X_{11}|^{4+\delta}\le \mu_{4+\delta}$ and 
	$$
	|X_{11}|\le A(Np)^{\frac12-\varkappa}\ln N\quad\text{  a.s.}
	$$
\end{itemize}
We use the following  result of Seginer (see \cite[Corollary 2.2]{Seginer}).
\begin{prop}There exists a constant $A$ such that for any $N,n\ge 1$, any $q\le 2\log\max\{n,N\}$, and any $N\times n$ random matrix $\mathbf X=(X_{ij})$ where $X_{ij}$
are i.i.d. zero mean random variables, the following inequality holds:
\begin{align}
\max\Big\{\E\max_{1\le i\le N}\|\mathbf X_{i\cdot}\|_2^q,&\E\max_{1\le j\le n}\|\mathbf X_{\cdot j}\|_2^q\Big\}\notag \le \E\|\mathbf X\|^q\notag \\ &\le (2A)^q\Big(\E\max_{1\le i\le N}\E\|\mathbf X_{i\cdot}\|_2^q+\max_{1\le j\le n}\|\mathbf X_{\cdot j}\|_2^q\Big).
\end{align}
Here $\mathbf X_{i\cdot}$, resp. $\mathbf X_{\cdot j}$, denote the $i$-th row,  resp. the $j$-th column of $\mathbf X$.
\end{prop}
\begin{proof}[Proof of Theorem \ref{lambda}]Note that $s_1=\|\mathbf X\|$. {Using the notations introduced above, we now estimate $\E\|\mathbf X_{i\cdot}\|^q$.
 By the definition of }$\mathbf X$ we have
\begin{equation}\label{1}
\E\|\mathbf X_{i\cdot}\|_2^q=\E\Big(\sum_{k=1}^nX_{ik}^2\xi_{ik}\Big)^{\frac q2}\le 2^{q-1}\Big(\sum_{k=1}^n\E X_{ik}^2\xi_{ik}\Big)^{\frac q2}+2^{q-1}\E\Big|\sum_{k=1}^n(X_{ik}^2-1)\xi_{ik}\Big|^{\frac q2}.
\end{equation}
Note that
\begin{equation}\label{2}
\E X_{ik}^2\xi_{ik}=p.
\end{equation}
Now, applying Rosenthal's inequality we get
\begin{align}\label{3}
\E\Big|\sum_{k=1}^n(X_{ik}^2-1)\xi_{ik}\Big|^{\frac q2}\le C^q\Big(q^{\frac q4}\Big(\sum_{k=1}^n\E(X_{ik}^2-1)^2\xi_{ik}\Big)^{\frac q4}+q^{\frac q2}p\sum_{k=1}^n\E|X_{ik}^2-1|^{\frac q2}\Big),
\end{align}
which implies 
\begin{equation}\label{4}
\E\Big|\sum_{k=1}^n(X_{ik}^2-1)\xi_{ik}\Big|^{\frac q2}\le C^q\big(q^{\frac q4}(Np)^{\frac q4}+q^{\frac q2}Np\E|X_{11}|^q\big).
\end{equation}
By assumptions $(CI)$, we have
\begin{equation}\label{5}
\E|X_{11}|^q\le C^q(Np)^{\frac q2-q\varkappa-2}\ln^{q-4-\delta} N.
\end{equation}
Note that for $q\sim \ln N$ inequality \eqref{5} coincide with condition \eqref{moment}.
Combining inequalities \eqref{1}--\eqref{5}, we now get
\begin{align*}
\E\|\mathbf X_{i\cdot}\|^q_2\le C^q(Np)^{\frac q2}\Big(1+\left(\frac{q}{Np}\right)^{\frac q4}+
N^{-1}p^{-1}\ln^{-(4+\delta)}N\left(\frac{q\ln^2 N}{(Np)^{2\varkappa}}\right)^\frac q2\Big).
\end{align*}
Taking into account  \eqref{np}, {as well as} $q\le C\log n$, we obtain, for $q\le 2\log\max\{n,N\}$,
\begin{equation*}
\E\|\mathbf X_{i\cdot}\|^q_2\le C^q(Np)^{\frac q2}.
\end{equation*} 
{A similar bound holds} for $\E\|\mathbf X_{\cdot j}\|^q$.
We may now write 
\begin{equation*}
\E\|\mathbf X\|^q\le C^qN(Np)^{\frac q2}.
\end{equation*}
Taking {$K\gg C$} and applying Markov's inequality, {the claim follows}.
Thus Theorem \ref{lambda} is proved.
\end{proof}
 \section
 {{Smallest singular values} }

 We shall now prove  Theorem \ref{minmain}  using an approach developed by Litvak, Pajor, Rudelson \cite{LPR}, Rudelson and Vershynin in \cite{RudVer} for rectangular matrices for the case $p=1$ and G\"otze and Tikhomirov in \cite{GT:2010} for the sparse dilute Wigner matrices. Denote by $\mathcal S^{(n-1)}$  the unit sphere in $\mathbb R^{n}$.
  Let $\mathbf x=(x_1,\ldots,x_n)\in\mathcal S^{(n-1)}$ be a fixed unit vector and $\mathbf X$ be a matrix defined in \eqref{sparerectmat}.

We divide the vectors on the sphere into two parts: compressible and incompressible vectors recalling the definition.
 \begin{df}
  Let $\delta,\rho\in(0,1)$. 
A vector $\mathbf x\in\mathbb R^n$ is called sparse if $|{\rm supp}(\mathbf x)|\le\delta n$. A vector $\mathbf x\in\mathcal S^{(n-1)}$ is called compressible if $\mathbf x$ is within Euclidean distance $\rho$ from the set of all sparse vectors. A vector $\mathbf x\in\mathcal S^{(n-1)}$ is called incompressible if it is not compressible. The sets of compressible and incompressible vectors will be denoted by $\text{\it Comp}(\delta,\rho)$ and $\text{\it Incomp}(\delta,\rho)$.
\end{df}
Note that
\begin{equation*}
s_n=\inf_{\mathbf x\in\mathcal S^{(n-1)}}\|\mathbf X\mathbf x\|_2
\end{equation*}
{and}
\begin{equation}\label{in1}
\Pr\{s_n\le \tau\sqrt{Np}\}\le \Pr\{\inf_{x\in\text{\it Comp}(\delta,\rho)}\|\mathbf X\mathbf x\|_2\le \tau\sqrt{Np}\}+\Pr\{\inf_{x\in\text{\it Incomp}(\delta,\rho)}\|\mathbf X\mathbf x\|_2\le \tau\sqrt{Np}\},
\end{equation}
for some $\delta,\rho\in(0,1)$ and $\tau>0$, { not depending} on $n$.

For  sparse matrices with $p=p_N\to 0$ as $N\to\infty$ {we cannot  directly estimate the first term on the right hand side of \eqref{in1} using  
the well-known two step approach of estimating
$\Pr\{\|\mathbf X\mathbf x\|_2\le \tau\sqrt{Np}\}$ for a  fixed vector
$\mathbf x\in\mathcal S^{(n-1)}$ followed by a  union bound for the some $\varepsilon$-net of $\text{\it Comp}(\delta,\rho)$ and  arriving at a bound for the infimum of $\mathbf x\in \text{\it Comp}(\delta_n,\rho)$ with $\delta_n\sim p$ going to zero.
 The Rudelson - Vershynin  methods for incompressible vectors won't work in this case.
In order to estimate $\Pr\{\inf_{x\in\text{\it Comp}(\delta,\rho)}\|\mathbf X\mathbf x\|_2\le \tau\sqrt{Np}\}$ with some $\delta>0$ which does not not depend on $n$, we shall use a method  developed in  G\"otze-Tikhomirov \cite{GT:2010}. This is based on  a recurrence approach which allows us to increase $\delta_N$ step by step  $Np$ times arriving   in $\log N$ steps at an estimate  of $\delta>\delta_0$ which does not depend on $N$. The details of this approach  will be described  in Section \ref{s3}.

In Section \ref{s4} we {shall derive bounds for $\Pr\{\inf_{x\in\text{\it Incomp}(\delta,\rho)}\|\mathbf X\mathbf x\|_2\le \tau\sqrt{Np}\}$.}
\subsection{Compressible vectors}\label{s3}Let $L$ {be an integer such that}
\begin{equation}\label{cond1}
\left(\frac{ \delta_0Np}{|\log p|+1}\right)^{L-1}\le p^{-1}\le \left(\frac{\delta_0 Np}{|\log p|+1}\right)^{L},
\end{equation}
where $\delta_0\in(0,1)$ {denotes} some constant independent on $N$.
Note that under {the} conditions of Theorem \ref{minmain} 
\begin{equation}
L\le c\log N/\log\log N
\end{equation}
with {a} constant $c=c(\delta_0).$
 We introduce {a}  set of numbers $p_{\nu N}$ and $\delta_{\nu N}$, for $\nu=1,\ldots , L$, as follows
\begin{equation*}
p_{\nu N}=(Np)\delta_{\nu-1 N} \text{ and  } \delta_{\nu N}=\delta_0 p_{\nu N}/(1+|\log p_{\nu N}|).
\end{equation*}
Here
\begin{equation*}
p_{0N}=p\text{ and }\delta_{0N}=\delta_0p/(1+|\log p|).
\end{equation*}
{Furthermore, introduce as well }
\begin{equation*}
\widehat p_{\nu N}=\left(\frac{Np\delta_0}{|\log p|+1}\right)^{\nu}p\text{ and }\widehat\delta_{\nu N}:=\left(\frac{\delta_0Np}{|\log p|+1}\right)^{\nu-1}\frac{\delta_0p}{|\log p|+1}.
\end{equation*}
\begin{lem}
The following inequalities hold
\begin{equation}
p_{\nu,N}\ge \widehat p_{\nu}
\end{equation}
and 
\begin{equation}
\delta_{\nu,N}\ge \widehat\delta_{\nu,N},
\end{equation}
for $\nu=1,\ldots,N$
\end{lem}
\begin{proof} By condition of Theorem \ref{minmain}, 
\begin{equation}
\frac{Np}{1+|\ln p|}\ge B\ln N.
\end{equation}
Without loss of generality we may assume that 
\begin{equation}
\frac{Np\delta_0}{1+|\ln p|}>1.
\end{equation}
It is straightforward to check now that $p_{\nu,N}\ge p$, for $\nu=1,\ldots,N$.  In fact,  for $\nu=1$ it is easy. Assume that for some $\nu=1,\ldots,N-1$ the inequality $p_{\nu-1,N}\ge p$ holds.
Then
\begin{equation}
p_{\nu,N}=\frac{Np\delta_0p_{\nu-1,N}}{1+|\ln p_{n-1,N}|}\ge \frac{Np\delta_0}{1+|\ln p|}p_{\nu-1,N}\ge\frac{Np\delta_0}{1+|\ln p|}p\ge p.
\end{equation}
We may write now the following inequalities
\begin{equation}
\delta_{\nu,N}\ge \frac{\delta_0}{1+|\ln p|}p_{\nu,N}
\end{equation}
and
\begin{equation}
p_{\nu,N}\ge \frac{Np\delta_0}{1+|\ln p|}p_{\nu-1,N},
\end{equation}
for $\nu=1,\ldots,N$.
Applying induction for the last inequality, we get, for $\nu=1,\ldots, N$,
\begin{equation}
p_{\nu,N}\ge\widehat p_{\nu,N}.
\end{equation}
The last inequality implies that, for $\nu=1,\ldots,N$,
\begin{equation}
\delta_{\nu,N}\ge \frac{\delta_0}{1+|\ln p|}\widehat p_{\nu-1,N}=\left(\frac{Np\delta_0}{1+|\ln p|}\right)^{\nu-1}\frac{p\delta_0}{1+|\ln p|}=\widehat\delta_{nu,N}.
\end{equation}
Thus, lemma is proved.
\end{proof}
\begin{cor}\label{ineq}There exist  constants $\gamma_0>0, \gamma_1>0$ such that
\begin{equation}\label{gamma_0}
\delta_{L,N}\ge\gamma_0\text{  and }p_{LN}\ge \gamma_1.
\end{equation}
\end{cor}


Introduce the sets
\begin{equation*}
\mathcal C_{\nu}:=\text{Comp}(\delta_{\nu,N},\rho), \quad{\mathcal IC}_{\nu}:=\text{Incomp}(\delta_{\nu,N},\rho),\quad \nu=0,\ldots,L.
\end{equation*}
Note  that $L\ge 1$ for $Np^2/(|\log p|+1)\le D$ with some constant $D$.
The case $Np^2/(|\log p|+1)\ge D$ {will we treated} separately.
In what follows we shall assume that $L\ge 1$.
{\begin{df}The L\'evy concentration function of a random variable $\xi$ is defined for $\varepsilon>0$ as
\begin{equation}
\mathcal L(\xi,\varepsilon)=\sup_{v\in\mathbb R}\Pr\{|\xi-v|\le\varepsilon\}.
\end{equation}
\end{df}}

{By $\mathbf P_{\mathbb E}$ we denote the orthogonal projection in $\mathbb R^n$ onto a subspace $\mathbb E$. Similarly, by $\mathbf P_{\mathbb J}$ we denote the orthogonal projection onto $\mathbb R^{\mathbb J}$, where $\mathbb J\subset\{1,2,\ldots,n\}$. }
 
 We reformulate and prove some auxiliary results from \cite{RudVer} below for our sparsity model.

First we prove an analog of \cite[Lemma 3.2]{RudVer}.
\begin{lem}\label{concentr} Let $\mathbf x\in{\mathcal {IC}}_{\nu}$, $\nu=1,\ldots, L$.
Let 
$$
\zeta_j=\sum_{k=1}^nx_k\xi_{jk}X_{jk},\quad j=1,\ldots,N.
$$
Then there exists some absolute constant $A$ such that
\begin{equation}
\mathcal L(\frac1{\sqrt p}\zeta_j, \frac{\rho}2)\le1- A\rho^4p_{\nu N}.
\end{equation}
\end{lem}
\begin{rem}\label{key} For $\nu=L$ there exists some constant $0<b<1$ such that 
\begin{equation*}
\mathcal L(\frac1{\sqrt p}\zeta_j, \frac{\rho}2)\le 1-b<1.
\end{equation*}
\end{rem}
\begin{proof}By Lemma \ref{sig} there exists a set $\sigma(x)$ such that for $k\in\sigma(\mathbf x)$
\begin{align*}
\frac1{2\sqrt n}\le |x_k|\le \frac1{\sqrt{2n\delta_{\nu-1, N}}}, \text{ and } \|{\mathbf P}_{\sigma(\mathbf x)}\mathbf x\|_2^2\ge \rho^2.
\end{align*}
Let
$$
\eta=\sum_{k\in\sigma(\mathbf x)}x_k\xi_{jk}X_{jk}/\sqrt p.
$$
Note that
\begin{equation*}
\E\eta^2\ge \rho^2,\quad \E|\eta|^4\le A_0(1+\frac1{N\delta_{\nu-1,N}p}).
\end{equation*}
Without loss of generality we may assume that $N\delta_{\nu-1,N}p\le 1$. This implies that
\begin{equation}\label{eta4}
 \E|\eta|^4\le \frac{2A_0}{N\delta_{\nu-1,N}p}.
\end{equation}
Let $Z=\eta-v$. Note that
\begin{equation*}
\E Z^2=\E\eta^2+v^2\ge v^2+\rho^2,
\end{equation*}
and
\begin{equation*}
\E\eta^4\ge(\E\eta^2)^2\ge\rho^4.
\end{equation*}
Using {Minkowski's} inequality, we get
\begin{equation*}
\E^{\frac14}|Z|^4\le\E^{\frac14}|\eta|^4+v\le\E^{\frac14}|\eta|^4(1+\frac v{\rho})\le \rho^{-1}\sqrt2\E^{\frac14}|\eta|^4(\rho^2+v^2)^{\frac12}.
\end{equation*}
Using {the Paley-Zygmund} inequality, we get
\begin{equation*}
\Pr\{|\eta-v|>\varepsilon\}\ge \frac{\rho^4(\E |Z|^2-\varepsilon^2)^2}{4\E|\eta|^4(\rho^2+v^2)^2}\ge \frac1{4\E|\eta|^4}\frac{\rho^4(\rho^2+v^2-\varepsilon^2)^2}{(\rho^2+v^2)^2}.
\end{equation*}
The last inequality and inequality \eqref{eta4} together imply
\begin{equation*}
\Pr\{|\eta-v|\ge\varepsilon\}\ge A_1\rho^4N\delta_{\nu-1,N}p(1-\frac{2\varepsilon^2}{\rho^2+v^2}).
\end{equation*}
Finally, we may write
\begin{equation*}
\Pr\{|\eta-v|\ge\frac12\rho\}\ge \frac12A_1\rho^4p_{\nu,N}.
\end{equation*}
Thus Lemma \ref{concentr} is proved.
\end{proof}

{ For the set of sparse vectors the following lemma holds.}
\begin{lem}\label{concsingle}
The following inequality holds.
\begin{equation*}
\mathcal L(\xi X/\sqrt p,\frac12)\le 1-\frac p{8\mu_4}
\end{equation*}
\end{lem}
\begin{proof}For the proof it is enough to note that {by the Paley-Zygmund inequality} we have
\begin{equation*}
\Pr\{|\xi X-v|\ge \frac12\}\ge p \frac{1+v^2-\varepsilon^2}{4\E|X|^4(1+v^2)^2}\ge \frac p{8\mu_4}
\end{equation*}
\end{proof}

\begin{lem}\label{tensorization} Let $\zeta_1,\ldots,\zeta_N$ {denote} independent identically distributed random variables such that
\begin{equation*}
\Pr\{|\zeta_j|\le \lambda_n\}\le 1-q_N,
\end{equation*}
for some $\lambda_N>0$ and $q_N\in(0,1)$. {Then there exist constants $c,C$ such that
\begin{equation}
\Pr\{\sum_{j=1}^N\zeta_j^2\le CNq_N\lambda_N^2\}\le \exp\{-cNq_N\}.
\end{equation}}
\end{lem} 
For the proof of this lemma see \cite[Lemma 4.5]{GT:2010}.

  We start with the estimation of $\|\mathbf X\mathbf x\|_2$ for a fixed $\mathbf x\in\mathcal S^{(n-1)}$. 
 \begin{lem}\label{firstfixpoint}
 There exist positive absolute constants $\tau_0$ and $c_0$ such that
 \begin{equation*}
 \Pr\{\|\mathbf X\mathbf x\|_2\le\tau_0\sqrt{Np}\}\le \exp\{-c_0Np\}.
 \end{equation*}
 \end{lem}
 \begin{proof}[Proof of Lemma  \ref{firstfixpoint}] The proof of this lemma {may be found  in \cite[Lemma 4.1]{GT:2010}, but for readers convenience  we repeat it here.} Let
 \begin{equation*}
 \zeta_j=\sum_{k=1}^nX_{jk}\xi_{jk}x_k,\quad j=1,\ldots,N
 \end{equation*}
 Then
 \begin{equation*}
 \|\mathbf X\mathbf x\|_2^2=\sum_{j=1}^N\zeta_j^2.
 \end{equation*}
 Furthermore, we may write for $\tau>0$ and any $t$
 \begin{align*}
 \Pr\{\sum_{j=1}^N\zeta_j^2\le\tau^2 Np\}=&\Pr\{\frac{\tau^2Np}2-\frac12\sum_{j=1}^N\zeta_j^2\ge0\}\le
 \exp\{Np\tau^2t^2/2\}\prod_{j=1}^N\E\exp\{-t^2\zeta_j^2/2\}.
 \end{align*}
 Using ${\rm e}^{-t^2/2}=\E{\rm e}^{it\eta}$, where $\eta$ is a standard Gaussian random variable,
 we obtain
 \begin{align}\label{annals}
 \Pr\{\sum_{j=1}^N\zeta_j^2<\tau^2 np\}\le \exp\{Np\tau^2t^2/2\}\prod_{j=1}^N\E_{\eta_j}
 \prod_{k=1}^n\E_{\xi_{jk}X_{jk}}\exp\{it\xi_{jk}X_{jk}x_k\eta_j\},
 \end{align}
 where $\eta_j$, $j=1,\ldots,N$ denote i.i.d. Gaussian  standard r.v.s and $\E_Z$ denotes expectation with respect to $Z$ conditional on all other r.v.s.
 
 Take $\alpha=\Pr\{|\eta_1|\le C_1\}$ for some absolute positive constant $C_1$ which will be chosen later. Then it follows from \ref{annals} that
 \begin{align*}
\Pr\{\sum_{j=1}^N\zeta_j^2<\tau^2Np\}&
\le \exp\{t^2\tau^2Np/2\}\notag\\&\times\prod_{j=1}^N\Big(\alpha\Big|\E_{\eta_j}\Big\{\prod_{k=1}^n\E_{\xi_{jk}X_{jk}}\exp\{it\eta_jx_kX_{jk}\xi_{jk}\}\Big||\eta_j|\le C_1\Big\}\Big|+1-\alpha\Big).
 \end{align*}
 Note that for any $\alpha, x\in[0,1]$, and $\beta\le\alpha$
 \begin{equation*}
 1-\alpha +\alpha x\le \max\{x^\beta,\Big(\frac{\beta}{\alpha}\Big)^{\frac{\beta}{1-\beta}}\}.
 \end{equation*}
Furthermore, we have
\begin{align}\label{166}
|\E_{\xi_{jk}X_{jk}}\exp\{it\xi_{jk}X_{jk}x_k\eta_j\}|&\le\exp\{-\frac p2(1-|f_{jk}(tx_k\eta_j)|^2)\},
\end{align}
where $f_{jk}(u)=\E\exp\{iuX_{jk}\}$.
 Choose {a}  constant $M>0$ such that
 \begin{equation*}
 \sup_{j,k\ge1}\E|X_{jk}|^2\mathbb I\{|X_{jk}|>M\}\le \frac12.
 \end{equation*}
 Since $1-\cos x\ge \frac{11}{24}x^2$ for $|x|\le1$, conditioning on the event $|\eta_j|\le C_1$, we get for $|t|\le \frac1{MC_1}$,
 \begin{equation}\label{2.11}
 1-|f_{jk}(tx_k\eta_j)|^2=\E_{X_{kj}}(1-\cos(tx_k\widetilde X_{kj}\eta_j)\ge\frac{11}{24}x_k^2t^2\eta_j^2\E|\widetilde X_{kj}|^2\mathbb I\{|X_{kj}|\le M\}.
 \end{equation}
 Here we denote by $\widetilde X_{kj}$ {the symmetrization of the r.v.} $X_{kj}$. 
 It follows from \eqref{166} for $|t|\le 1/(MC_1)$, {that  for} $|\eta_j|\le C_1$,
 \begin{equation}\label{2.12}
 |\E_{\xi_{jk}X_{jk}}\exp\{it\xi_{jk}X_{jk}x_k\eta_j\}|\le \exp\{-cpt^2x_k^2\eta_j^2\}
 \end{equation}
 This implies that
 \begin{equation}\label{2.13}
 |\prod_{k=1}^n\E_{\xi_{kj}X_{kj}}\exp\{it\eta_jx_k\xi_{jk}X_{jk}\}|\le \exp\{-c p t^2\eta_j^2\}.
 \end{equation}
 We may choose $C_1$ large enough such that following inequalities hold for $|t|\le 1/MC_1$:
 \begin{equation}\label{2.14}
 |\E_{\eta_j}\{\exp\{-cpt^2\eta_j^2\}\big||\eta_j|\le C_1\}|\le \exp\{-ct^2p/24\}.
 \end{equation}
 Then we obtain
 \begin{align}\label{2.15}
 \Pr\{\sum_{j=1}^N\zeta_j^2\le \tau^2Np\}\le \exp\{Np\tau^2t^2/2\}\Big(\exp\{-c\beta t^2Np/24\}+\Big(\frac{\beta}{\alpha}\Big)^{N\frac{\beta}{1-\beta}}\Big)
 \end{align} 
 Furthermore, we may take $C_1$ {sufficiently  large  such that} $\alpha\ge \frac45$ and choose $\beta=\frac25$. We get
 \begin{align}\label{2.16}
  \Pr\{\sum_{j=1}^N\zeta_j^2\le \tau^2Np\}\le \exp\{Np\tau^2t^2/2\}\Big(\exp\{-c t^2Np/60\}+2^{-2N/3}\Big).
 \end{align}\label{2.17}
   For $\tau<\min\{\frac{\sqrt c}{\sqrt{60}},\frac{\sqrt{\ln 2}}{\sqrt 3}MC_1\}$, we have for $|t|\le 1/(MC_1)$,
 \begin{equation}\label{2.18}
 \Pr\{\sum_{j=1}^N\zeta_j^2\le \tau^2Np\}\le\exp\{-ct^2Np/120\}.
 \end{equation}
 This implies the claim.
 Thus the  lemma is proved.
 \end{proof}
 \subsection{{Compressible and Incompressible Vectors}}

First we prove an analog of Lemma 2.6 from \cite{RudVer}.

\begin{lem}	\label{comp}There exist positive absolute constants $\delta_0,\tau_0, c_1$ such that
\begin{equation*}
\Pr\{\inf_{\mathbf x\in {\text{\it Comp}(\delta_{0N},\rho_0)}}\|\mathbf X\mathbf x\|_2\le\tau_0\sqrt{Np}, \quad\|\mathbf X\|\le K\sqrt{Np}\}\le \exp\{-c_1Np\},
\end{equation*}
where
\begin{align}\label{delta0n}
\delta_{0N}=\delta_0p/(|\log p|+1),\quad \rho_0=\tau_0/2K.
\end{align}
\end{lem}
\begin{proof} Let $k=[n\delta_{0N}]$. Denote by $\mathcal N_{\eta}$ {an}  $\eta$-net on  the $\mathcal S^{(k-1)}\cap \mathbb R^{k}$. Choose $\eta=\tau_0/2K$
First we consider the set of all sparse vectors ${\it Sparse}(k)$ with ${\rm support}(\mathbf x)\le k$. Using Lemma \ref{firstfixpoint} and {a} union bound, we get
\begin{equation*}
\Pr\{\inf_{\mathbf x\in {\it Sparse}(\delta_{0N})}
\|{\mathbf X}\mathbf x\|_2
\le2\rho_0\sqrt{np}\}
\le \binom nk|\mathcal N_{\eta}|\exp\{-c_0Np\}.
\end{equation*}
Using Stirling's formula and Proposition 2.1 from \cite{RudVer}, we get
\begin{align*}
\Pr\{\inf_{\mathbf x\in {\it Sparse}(\delta_{0N})}
\|{\mathbf X}\mathbf x\|_2&
\le2\tau_0\sqrt{Np}\}\notag\\&
\le \frac {4n\delta_{0N}}{\sqrt{2\pi n\delta_{0N}(1-\delta_{0N})}}\frac{(1+\frac{K}{\rho_0})^{n\delta_{0N}-1}}{\delta_{0N}^{n\delta_{0N}}(1-\delta_{0N})^{n(1-\delta_{0N})}}\exp\{-c_0Np\}.
\end{align*}
Simple calculations show
\begin{align*}
\Pr&\{\inf_{\mathbf x\in {\it Sparse}(\delta_{0N})}
\|{\mathbf X}\mathbf x\|_2
\le2\tau_0\sqrt{Np}\}\le \sqrt{\frac{2n\delta_{0N}}{(1-\delta_{0N})\pi}}\notag\\&\times\exp\{n\delta_{0N}\Big((1-\frac 1{n\delta_{0N}})\frac K{\rho_0}-\log\delta_{0N}-(1-\delta_{0N})\frac1{\delta_{0N}}\log(1-\delta_{0N})\Big)-c_0Np\}.
\end{align*}
If we {choose}
\begin{equation*}
\delta_{0N}:=\delta_0 p/(1+|\log p|)
\end{equation*}
for {a} sufficiently small absolute constant $\delta_0$, we get

\begin{equation*}
\Pr\{\inf_{\mathbf x\in {\it Sparse}(\delta_{0N})}
\|{\mathbf X}\mathbf x\|_2
\le2\tau_0\sqrt{Np}\}\le \exp\{-c_1Np\}.
\end{equation*}
{
Thus the Lemma is proved.}
\end{proof}
In what follows,  
{we shall use a technique developed in G\"otze and Tikhomirov \cite{GT:2010} which is based on the following lemmas.}
\begin{lem}\label{sig} Let $\rho,\delta\in(0,1)$. Assume that $\mathbf x\in\text{Incomp}(\delta,\rho)$. Then there exists a set $\sigma_0(x)$ such that $|\sigma_0(x)|\ge Cn\delta\rho^2$  and
$\frac1{2\sqrt n}\le |x_k|\le \frac1{\sqrt{n\delta/2}}$ for $k\in\sigma_0(x)$,  and 
$$
\sum_{k\in\sigma_0(x)}|x_k|^2\ge\rho^2.
$$

\end{lem}
For a  proof of this Lemma see for instance  \cite[Lemma 3.4]{RudVer2008}.

\begin{lem}\label{2.6}Let $\mathbf x\in{\mathcal IC}_{\nu}$ for some $\nu=0,\ldots,L-1$. Then there exist constants $c_1$ and $c_2$ such that for any $0<\tau\le \tau_0$
\begin{equation*}
\Pr\{\|\mathbf X\mathbf x\|_2\le\tau\sqrt{Np}\}\le\exp\{-c_1Np_{\nu+1 N}\}.
\end{equation*}
\end{lem}
\begin{proof}  We repeat the proof of Lemma \ref{firstfixpoint} till \eqref{2.11}.

Furthermore, by Lemma \ref{sig} there exists a set $\sigma_0(x)$ such that $\frac1{2\sqrt n}\le |x_k|\le \frac1{\sqrt{n\delta_{\nu N}/2}}$ for $k\in\sigma_0(x)$,  and 
\begin{equation}\label{rho}
\sum_{k\in\sigma_0(x)}|x_k|^2\ge\rho^2.
\end{equation}
We may write now
\begin{align*}
\sum_{k=1}^n(1-|f(tx_kX_{jk}\eta_j)|^2)\ge\sum_{k\in\sigma_0(x)}(1-|f(tx_kX_{jk}\eta_j)|^2).
\end{align*}
Note that for $k\in\sigma_0$, and for $|X_{jk}|\le M$, and for $|\eta_j|\le C$, we have
$$
|tx_kX_{jk}\eta_j|\le \frac{|t|CM\sqrt2}{\sqrt{N\delta_{\nu N}}}.
$$
Taking  $t={\kappa}{\sqrt{N\delta_{\nu N}}}$ for  $\kappa=\frac1{CM\sqrt 2}$, we get
\begin{equation*}
|tx_kX_{jk}\eta_j|\le 1,
\end{equation*}
and 
\begin{equation*}
1-|f_{\eta_j}(tx_kX_{jk}\eta_j)|^2\ge \frac{11}{24}t^2x_k^2\eta_j^2\E|X_{jk}|^2\mathbb I\{|X_{jk}|\le M\}\ge \frac{11}{48}t^2x_k^2\eta_j^2.
\end{equation*}
Repeating now the {last part of the proof of Lemma \ref{firstfixpoint} and taking into account inequality \eqref{rho}}, we obtain  for $\tau<\rho\min\{\frac{\sqrt c}{\sqrt{60}},\frac{\sqrt{\ln 2}}{\sqrt 3}MC_1\}$, and for $|t|= \kappa\sqrt{N\delta_{\nu N}}$,

 \begin{equation}\label{3.13}
 |\prod_{k=1}^n\E_{\xi_{jk}X_{jk}}\exp\{it\eta_jx_k\xi_{jk}X_{jk}\}|\le \exp\{-c\rho^2pt^2\eta_j^2\},
 \end{equation}
 where $c$ is an absolute constant as in \eqref{2.13}.  
 We may choose $C_1$ large enough such that the following inequalities hold for $|t|=\kappa\sqrt{N\delta_{\nu N}}$:
 \begin{equation}\label{2.14}
 |\E_{\eta_j}\{\exp\{-cpt^2\eta_j^2\}\big||\eta_j|\le C_1\}|\le \exp\{-ct^2p/24\}.
 \end{equation}
 We use here that $|t|p\le \delta_0$ by \eqref{cond1}.
 Then we obtain
 \begin{align}\label{2.15}
 \Pr\{\sum_{j=1}^n\zeta_j^2\le \tau^2Np\}\le \exp\{Np\tau^2t^2/2\}\Big(\exp\{-c\beta t^2Np/24\}+\Big(\frac{\beta}{\alpha}\Big)^{N\frac{\beta}{1-\beta}})\}
 \end{align} 
 Furthermore, we may take $C_1$ {large enough such that} $\alpha\ge \frac45$ and choose $\beta=\frac25$. We get
 \begin{align}\label{2.16}
  \Pr\{\sum_{j=1}^n\zeta_j^2\le \tau^2Np\}\le \exp\{Np\tau^2t^2/2\}\Big(\exp\{-c t^2Np/60\}+2^{-2N/3}\Big).
 \end{align}\label{2.17}
   For $\tau<\min\{\frac{\sqrt c}{\sqrt{60}},\frac{\sqrt{\ln 2}}{\sqrt 3}MC_1\}$, we have for $|t|=\kappa\sqrt{N\delta_{\nu N}}$,
 \begin{equation}\label{2.18}
 \Pr\{\sum_{j=1}^n\zeta_j^2\le \tau^2Np\}\le\exp\{-ct^2Np/120\}.
 \end{equation} 
 This inequality implies that
 \begin{equation}\label{2.18}
 \Pr\{\sum_{j=1}^N\zeta_j^2\le \tau^2Np\}\le\exp\{-c(\rho^2N^2\kappa^2p\delta_{\nu N}\wedge N)/120\}.
 \end{equation}
 { Thus the lemma is proved.}
\end{proof}

Furthermore, we consider the sets defined as
\begin{equation}
\widehat {\mathcal C}_{\nu}:={\mathcal IC}_{\nu-1}\cap {\mathcal C_{\nu}}, \ \nu=1,\ldots, L.
\end{equation}
\begin{lem}\label{finsparse} Under conditions of Theorem \ref{minmain} we have, for $\nu=1,\ldots,L$,
\begin{equation*}
\Pr\{\inf_{\mathbf x\in\widehat{\mathcal C}_{\nu}}\|\mathbf X\mathbf x\|_2\le \tau\sqrt{Np}\}\le \exp\{-cNp_{\nu N}\}.
\end{equation*}
\end{lem}
\begin{proof}
According to Lemma \ref{2.6} we have for any fixed $\mathbf x\in\widehat{\mathcal C}_{\nu}$
\begin{equation*}
\Pr\{\|\mathbf X\mathbf x\|_2\le 2\tau\sqrt{Np}\}\le \exp\{-c_1Np_{\nu, N}\}.
\end{equation*}
Consider $\eta=\frac{\tau}{K}$-net $\mathcal N$ of $\widehat{\mathcal C}_{\nu}$. Then the event $ \{\inf_{\mathbf x\in\widehat{\mathcal C}_{\nu}}\|\mathbf X\mathbf x\|_2\le \tau\sqrt{Np}\}$ implies
\begin{equation}
 \{\inf_{\mathbf x\in\mathcal N}\|\mathbf X\mathbf x\|_2\le 2\tau\sqrt{Np}\}.
\end{equation}
Without loss of generality  we may assume that $\delta_{LN}<1$.
Using {a} union bound, we get
\begin{align}
\Pr\{\inf_{\mathbf x\in\widehat{\mathcal C}_{\nu}}\|\mathbf X\mathbf x\|_2\le \tau\sqrt{Np}\}\le \binom{n}{ n\delta_{\nu N}}|\mathcal N|\exp\{-c_1Np_{\nu, N}\}
\end{align}
Using Stirling's formula and {a} simple bound for the cardinality of {an} $\eta$-net,  for some  sufficiently small absolute constant $\alpha_0>0$ (does not depend on $\nu$) and
\begin{equation*}
\delta_{\nu N}=\alpha_0 p_{\nu N}/(|\log p_{\nu,N}|+1),\ p_{\nu N}:=Np\delta_{\nu-1,N}
\end{equation*}
we get
\begin{align*}
\Pr\{\inf_{\mathbf x\in\widehat{\mathcal C}_{\nu}}\|\mathbf X\mathbf x\|_2\le \tau\sqrt{Np}\}\le\exp\{-\widehat c_1Np_{\nu N}\}.
\end{align*}
Thus Lemma \ref{finsparse} is proved.
\end{proof}

Now we consider the case $Np^2/(|\log p|+1)>D$ for some sufficiently large constant $D$. Let $\mathbf x\in{\it Incomp}(\delta_{0N},\rho)$ and $\sigma(\mathbf x)$ {denote the set described} in Lemma \ref{sig}. 
Let 
\begin{equation*}
\zeta_j=\sum_{k=1}^nx_k\xi_{jk}X_{jk}, j=1,\ldots, N.
\end{equation*}
We have
\begin{equation*}
\mathcal L(\zeta_j,\tau\sqrt p)\le \mathcal L(\sum_{k\in\sigma(\mathbf x)}x_k\xi_{jk}X_{jk},\tau\sqrt p).
\end{equation*}
Using  {a} Berry-Esseen bound we get
\begin{equation*}
\mathcal L(\zeta_j,\tau\sqrt p)\le C\tau+C\frac{\sum_{k\in\sigma(x)}x_k^3p\E|X_{jk}|^3}{(\sum_{k\in\sigma(\mathbf x)}x_k^2p)^{\frac32}}\le C\tau+\frac{C\mu_3}{\rho\sqrt {n\delta_{0N}p}}.
\end{equation*}
Note that $np\delta_{0N}=y\delta_0Np^2/(1+|\ln p|)$.
Choosing $D$ sufficiently large, we {have}
\begin{equation*}
\mathcal L(\zeta_j,\tau\sqrt p)\le 1-b,
\end{equation*}
for some constant $b\in(0,1)$. By Lemma \ref{tensorization} we get
\begin{equation*}\label{big}
\Pr\{\|\mathbf X\mathbf x\|_2\le 2\tau\sqrt{Np}\}\le \exp\{-cN\},
\end{equation*} 
for $\tau\le \tau_0$ and $c>0$.

Inequality \eqref{big} implies that there exists $\gamma_0>0$ such that
\begin{equation*}
\Pr\{\inf_{\mathbf x\in \mathcal C_1\cap{\it Incomp}(\delta_0,\rho)}\|\mathbf X\mathbf x\|_2\le\tau\sqrt{Np}\}\le \exp\{-cN\}.
\end{equation*}

Note that
\begin{equation*}
\text{\it Comp}(\delta_{LN},\rho)\subset\mathcal C_{0}\cup\left(\cup_{\nu=1}^{L}\widehat{\mathcal C}_{\nu}\right).
\end{equation*}
Using {a} union bound, we get 
\begin{equation}\label{spatsefinal}
\Pr\{\inf_{x\in\text{\it Comp}(\delta_{LN},\rho)}\|\mathbf X\mathbf x\|_{2}\le\tau\sqrt{np}\}\le \exp\{-cNp\}+\sum_{\nu=1}^{L-1}
\exp\{-c(Np)^{\nu}N\delta_{0,N}\}\le \exp\{-\overline cNp\}.
\end{equation}
By Corollary \ref{ineq},
$$
\text{\rm Comp}(\gamma_0,\rho)\subset \mathcal C_L.
$$
This implies that
\begin{equation}\label{inc}
\inf_{\mathbf x\in\text{\it Incomp}(\gamma_0,\rho)}\|\mathbf X\mathbf x\|_{2}\le 
\inf_{\mathbf x\in\text{\it Incomp}(\delta_{LN},\rho)}\|\mathbf X\mathbf x\|_{2}.
\end{equation}
In what follows we shall estimate the probability 
$\Pr\{\inf_{\mathbf x\in\text{\it Incomp}(\gamma_0,\rho)}\|\mathbf X\mathbf x\|_{2}\le \tau\sqrt{Np}\}$.

\subsection{{Incompressible Vectors}}\label{s4} 

Using {a decomposition of the unit sphere $\mathbb S^{(n-1)}={\it Comp}\cup{\it Incomp}$,
we decompose  the invertibility problem onto two sub problems for compressible and incompressible vectors:}
\begin{align}\label{decomp}
\Pr\{s_n(\mathbf X)&\le \varepsilon \sqrt p\sqrt N\}
\notag\\&\le\Pr\{\inf_{\mathbf x\in{\it Comp}}\|\mathbf X\mathbf x\|_2\le\varepsilon \sqrt p\sqrt N\}+ \Pr\{\inf_{\mathbf x\in{\it Incomp}}\|\mathbf X\mathbf x\|_2\le\varepsilon\sqrt p\sqrt N\}.
\end{align}
A bound for the compressible vectors follows from inequality \eqref{spatsefinal}.
It remains to find a lower bound for $\|\mathbf X\mathbf x\|_2$ for incompressible vectors. 
Let $\eta,\eta_1,\ldots,\eta_N$ denote standard Gaussian random variables  independent of ${X_{jk},\xi_{jk}}$ for $1\le j\le N, 1\le k\le n$. We shall prove the following lemma. 
\begin{lem} \label{new}Let $x\in I\mathcal C(\delta,\rho)$. Then there exist absolute constants $c_1$ such that for any $C>0$ , the following inequality
\begin{equation}
\Pr\{\|\mathbf Xx\|_{2}\le t\sqrt{Np}\}\le (\frac {2t}{\sqrt{t^2+\rho^2/2}})^N+(\frac{2c_0}{C}\exp\{-\frac{C^2}2\})^N,
\end{equation}
holds for $t\ge c_1\mu_4/\sqrt{Np\delta}$. 
\end{lem} 
	\begin{proof}
	We may write
	\begin{align}
	\Pr\{\|\mathbf Xx\|_2\le t\sqrt{Np}\}=\Pr\{\sum_{j=1}^N\zeta_j^2<t^2Np\}
	\end{align}
	where $\zeta_j=\sum_{k=1}^nX_{jk}\xi_{jk}x_k$.
	Applying Markov's inequality, we get
	\begin{align}\label{first}
	\Pr\{\sum_{j=1}^N\zeta_j^2<t^2Np\}\le{\rm e}^N\E\exp\{-\frac1{t^2p}\sum_{j=1}^N\zeta_j^2\}=
	{\rm e}^N\prod_{j=1}^N\E\exp\{-\frac1{t^2p}\zeta_j^2\}.
	\end{align}
	We may rewrite the r.h.s. of \eqref{first} as follows
	\begin{align}\label{second}
	\Pr\{\sum_{j=1}^N\zeta_j^2<t^2Np\}\le{\rm e}^N\prod_{j=1}^N\E\exp\{i\frac1{t\sqrt p}\zeta_j\eta_j\}.
	\end{align}
Conditioning by $\eta_j$, we get
\begin{equation}
\Pr\{\sum_{j=1}^N\zeta_j^2<t^2Np\}\le{\rm e}^N\prod_{j=1}^N\E_{\eta_j}\prod_{k=1}^n|\E_{X_{jk}\xi_{jk}}\exp\{i\frac1{t\sqrt p}\eta_jx_kX_{jk}\xi_{jk}\}|
\end{equation}
By Lemma \ref{sig} there exists a set $\sigma(x)$ such that
for $k\in\sigma(x)$ we have $\frac1{2\sqrt n}\le |x_k|\le \frac {\sqrt 2}{ \sqrt{n\delta}}$ and 
 $|\sigma(x)|\ge \frac1{2y}\delta \rho^2N$.
We may write the following inequality
\begin{align}
\E_{\eta_j}\prod_{k\in\sigma(x)}|&\E_{X_{jk}\xi_{jk}}\exp\{i\frac1{t\sqrt p}\eta_jx_kX_{jk}\xi_{jk}\}|\notag\\&\le \E_{\eta_j}\prod_{k\in\sigma(x)}|\E_{X_{jk}\xi_{jk}}\exp\{i\frac1{t\sqrt p}\eta_jx_kX_{jk}\xi_{jk}\}|.
\end{align}
 For any constant $C$ we have 
\begin{align}
\E_{\eta_j}\prod_{k\in\sigma(x)}|&\E_{X_{jk}\xi_{jk}}\exp\{i\frac1{t\sqrt p}\eta_jx_kX_{jk}\xi_{jk}\}|\notag\\&\le \E_{\eta_j}\left(\prod_{k\in\sigma(x)}|\E_{X_{jk}\xi_{jk}}\exp\{i\frac1{t\sqrt p}\eta_jx_kX_{jk}\xi_{jk}\}|\right)\mathbb I\{|\eta_j|\le C\}+\Pr\{|\eta_j|>C\}.
\end{align}
Consider $k\in\sigma(x)$ now. Taking expectation with respect to $\xi_{jk}$ conditioning on $X_{jk}$ and $\eta_j$), we obtain 
\begin{align}
|\E_{X_{jk}\xi_{jk}}&\big(\exp\{i\frac1{t\sqrt p}\eta_jx_kX_{jk}\xi_{jk}\}\big)|\notag\\&= |1+p(\E_{X_{jk}}\exp\{i\frac1{t\sqrt p}\eta_jx_kX_{jk}\}-1)|.
\end{align}
Applying Taylor's formula for {the} characteristic function $\E_{X_{jk}}\exp\{i\frac1{t\sqrt p}\eta_jx_kX_{jk}\}$, we may write
\begin{align}
|1+p(\E_{X_{jk}}&\exp\{i\frac1{t\sqrt p}\eta_jx_kX_{jk}||\eta_j|\le C\}-1)|\notag\\&\le |1+p(-\frac1{2t^2p}\eta_j^2x_k^2+\frac{\E|X_{11}|^3}{6t^3p^{\frac32}}|x_k|^3|\eta_j|^3){|}.
\end{align}
Since $\E|X_{11}|^3\le \E^{\frac3{4}}|X_{11}|^{4}\le \mu_4^{\frac3{4}}\le \mu_4$, for $|\eta_j|\le C$, and
\begin{equation}
t\ge \frac {C\mu_4}{\sqrt{yNp\delta}},
\end{equation}
we have
$$
 \frac{|x_k||\eta_j|\E|X_{11}|^3}{3t\sqrt p}\le \frac{C\mu_4\sqrt2}{3t\sqrt {yN\delta p}}\le \frac1{2}.
 $$
 Taking into account this inequality, we get for $|\eta_j|\le C$,
\begin{equation}
|1+p(\E_{X_{jk}\xi_{jk}}\exp\{i\frac1{t\sqrt p}\eta_jx_kX_{jk}\}-1)|\le \exp\{-\frac 1{4t^2}x_k^2\eta_j^2\}.
\end{equation}
Since $\sum_{k\in\sigma(x)}x_k^2\ge \rho^2$, this inequality implies that
\begin{align}
\prod_{k=1}^n|\E_{X_{jk}\xi_{jk}}\exp\{i\frac1{t\sqrt p}\eta_jx_kX_{jk}\xi_{jk}\}|\mathbb I\{|\eta_j|\le C\}&\le \exp\{-\frac {\rho^2}{4t^2}\eta_j^2\}.
\end{align}

From here it follows for any $C>0$
\begin{align}
\Pr\{\sum_{j=1}^N\zeta_j^2<t^2Np\}\le\prod_{j=1}^N\big(\E\exp\{-\frac{\rho^2}{4t^2}\eta_j^2\}+\Pr\{|\eta_j|>C\}\big).
\end{align}
There exists an absolute constant $c_0>0$ such that
\begin{equation}
\Pr\{|\eta_j|>C\}\le \frac{c_0}{C}\exp\{-\frac{C^2}2\}.
\end{equation}
This inequality implies that 
\begin{align}
\Pr\{\sum_{j=1}^N&\zeta_j^2<t^2Np\}\le(\frac {t}{\sqrt{t^2+\rho^2/2}}+\frac{c_0}{C}\exp\{-\frac{C^2}2\})^N\notag\\&\le(\frac {2t}{\sqrt{t^2+\rho^2/2}})^N+(\frac{2c_0}{C}\exp\{-\frac{C^2}2\})^N.
\end{align}
Thus, Lemma \ref{new} is proved.
\end{proof}

\textbf{Proof of Theorem \ref{minmain}:}\\

First we note that
\begin{align}\label{3.58}
\Pr\{\inf_{\mathbf x\in\mathcal S^{(n-1)}}\|\mathbf X\mathbf x\|_2\le t\sqrt{Np}\}&\le \Pr\{\inf_{\mathbf x\in \text{Com}(\delta_{L,N},\rho)}\|\mathbf X\mathbf x\|_2\le t\sqrt{Np}\}\notag\\&+
\Pr\{\inf_{\mathbf x\in \text{Incomp}(\delta_{L,N},\rho)}\|\mathbf X\mathbf x\|_2\le t\sqrt{Np}\}.
\end{align}
By inequality \eqref{spatsefinal}, for some constant $\overline c>0$,
\begin{equation}\label{3.59}
\Pr\{\inf_{\mathbf x\in \text{Com}(\delta_{L,N},\rho)}\|\mathbf X\mathbf x\|_2\le t\sqrt{Np}\}\le \exp\{-\overline cNp\}.
\end{equation}
By Relation \eqref{inc}, we have
\begin{equation}
\Pr\{\inf_{\mathbf x\in \text{Incomp}(\delta_{L,N},\rho)}\|\mathbf X\mathbf x\|_2\le t\sqrt{Np}\}\le\Pr\{\inf_{\mathbf x\in \text{Incomp}(\gamma_0,\rho)}\|\mathbf X\mathbf x\|_2\le t\sqrt{Np}\}
\end{equation}
We consider an  $\varepsilon$-net $\mathcal N$ on the set of incompressible vectors $\mathcal {IC}(\gamma_0,\rho)$ with  {$\varepsilon=\frac{t}{2K}$  where $K>0$ is fixed}. It is straightforward to check that
\begin{equation}
\Pr\{\inf_{x\in\mathcal {IC}(\gamma_0,\rho)}\|\mathbf X\mathbf x\|_2\le \tau\sqrt{N{p}},\,\|\mathbf X\|\le K\sqrt{Np}\}\le\Pr\{\inf_{x\in\mathcal {N}}\|\mathbf X\mathbf x\|_2\le 2\tau\sqrt{N{p}}\}
\end{equation}
Applying {a union-bound}, we get 
\begin{equation}
\Pr\{\inf_{x\in\mathcal {N}}\|\mathbf X\mathbf x\|_2\le 2\tau\sqrt{N{p}}\}\le |\mathcal N|\sup_{x\in \mathcal{IC}(\gamma_0,\rho)}\Pr\{\|\mathbf X\mathbf x\|_2\le2\tau\sqrt{Np}\}.
\end{equation}

By \cite[Proposition2.1]{RudVer}, we have
$$
|\mathcal N|\le n\left(1+\frac2{\varepsilon}\right)^{n-1}.
$$
 Then, applying {the} result of Lemma \ref{new}, we get (for {$t\ge ...\frac{c_1\mu_4}{\sqrt{N\gamma_0p}}$)} 
\begin{align}\label{3.17}
\Pr\{\inf_{x\in\mathcal{IC}(\gamma_0,\rho)}&\|\mathbf Xx\|_2\le t\sqrt{Np}\}\le  |\mathcal N|\left((\frac {2t}{\sqrt{t^2+\rho^2/2}})^N+(\frac{2c_0}{C}\exp\{-\frac{C^2}2\})^N\right)\notag\\&\le yN\left(1+\frac{4K}t\right)^{n-1}\left((\frac {2t}{\sqrt{t^2+\rho^2/2}})^N+(\frac{2c_0}{C}\exp\{-\frac{C^2}2\})^N\right).
\end{align}
It is easy to see that, for any $0< t \le \tau_0$,
\begin{equation}
\Pr\{\inf_{x\in\mathcal{IC}(\delta,\rho)}\|\mathbf Xx\|_2\le t\sqrt{Np}\}\le \Pr\{\inf_{x\in\mathcal{IC}(\delta,\rho)}\|\mathbf Xx\|_2\le \tau_0\sqrt{Np}\}.
\end{equation}
Without loss of generality we may assume that $\tau_0\le 4K$. Taking into account both that $N\le {\rm e}^N$ and $y<1$
  rewrite the  inequality \eqref{3.17} in the form

\begin{align}\label{3.17}
\Pr\{\inf_{x\in\mathcal{IC}(\gamma_0,\rho)}\|\mathbf Xx\|_2\le\tau_0\sqrt{Np}\}&\le \left(\frac{5K}{2\tau_0}\right)^{yN}\left((\frac {4{\rm e}\tau_0}{\sqrt{4\tau_0^2+\rho^2/2}})^N+(\frac{2c_0{\rm e}}{C}\exp\{-\frac{C^2}2\})^N\right)\notag\\& \le\left(\left(\frac {(5K)^y4{\rm e}\sqrt 2}{\rho}\tau_0^{(1-y)}\right)^N+(\frac{2c_0{\rm e}(5K)^y}{C\tau_0^y}\exp\{-\frac{C^2}2\})^N\right)
\end{align}}
Put 
$$
\tau_0=\left(\frac{\rho}{4\sqrt 2\cdot5^y{\rm e}^2K^{y}}\right)^{\frac1{1-y}}.
$$
For $N\ge 2$, we have
\begin{equation}
\frac {(5K)^y4{\rm e}\sqrt 2}{\rho}\tau_0^{(1-y)}\
\le \frac12{\rm e}^{-\frac12N}.
\end{equation}
Note that, by condition \eqref{np}, for $N$ such that
\begin{equation}
\ln N\ge \frac{\mu_4}{\tau_0\sqrt{B\gamma_0}},
\end{equation}
we have 
\begin{equation}
\tau_0\ge\frac{\mu_4}{\sqrt{Np\gamma_0}}.
\end{equation}
Moreover,
 choosing $C$ such that
 $$
 C{\rm e}^{\frac{C^2}2}\ge \frac{2c_0{\rm e}5^yK^y}{2^y\rho^{\frac y{1-y}}\tau_0^y}
 $$

we obtain that

\begin{equation}\label{final}
\Pr\{\inf_{x\in\mathcal{IC}(\gamma_0,\rho)}\|\mathbf Xx\|_2\le t\sqrt{Np}, \|\mathbf X
\|\le K\sqrt{Np}\}\le {\rm e}^{-N/2},
\end{equation}
for any $0\le t\le \tau_0$.
The result of Theorem \ref{minmain} follows now from inequalities \eqref{3.58}, \eqref{3.59} and \eqref{final}. 
(Since $\gamma_0$ is an  absolute constant defined in Corollary \ref{ineq} .) Theorem \ref{minmain} is proved.



\end{document}